\documentclass[preprint,12pt]{elsarticle}

\biboptions{numbers,sort&compress}
\usepackage{amssymb}
\usepackage{amsthm}
\usepackage{amsmath}
\usepackage{epic}
\usepackage{CJK}
\usepackage{setspace}
\usepackage{bm}
\newtheorem{thm}{Theorem}[section]
\newtheorem{cor}[thm]{Corollary}
\newtheorem{lem}[thm]{Lemma}

\newtheorem{example}[thm]{Example}
\newtheorem{rem}{Remark}

\journal{}

\begin{document}
\begin{spacing}{1.15}
\begin{CJK*}{GBK}{song}
\begin{frontmatter}
\title{\textbf{Unified bounds for the independence number of graphs}}

\author{Jiang Zhou}\ead{zhoujiang@hrbeu.edu.cn}

\address{College of Mathematical Sciences, Harbin Engineering University, Harbin 150001, PR China}

\begin{abstract}
The Hoffman ratio bound, Lov\'{a}sz theta function and Schrijver theta function are classical upper bounds for the independence number of graphs, which are useful in graph theory, extremal combinatorics and information theory. By using generalized inverses and eigenvalues of graph matrices, we give bounds for independence sets and the independence number of graphs. Our bounds unify the Lov\'{a}sz theta function, Schrijver theta function and Hoffman-type bounds, and we obtain the necessary and sufficient conditions of graphs attaining these bounds. Our work leads to some simple structural and spectral conditions for determining a maximum independent set, the independence number, the Shannon capacity and the Lov\'{a}sz theta function of a graph.
\end{abstract}

\end{frontmatter}

\section{Introduction}
The \textit{independence number} $\alpha(G)$ of graph $G$ is the maximum size of independent sets in $G$, which is an important graph parameter in graph theory. The following is a celebrated upper bound given by Hoffman.
\begin{thm}\label{thm1.1}
\textup{[5, Theorem 3.5.2]} Let $G$ be a $k$-regular ($k\neq0$) graph with $n$ vertices, and let $\tau$ be the minimum adjacency eigenvalue of $G$. Then
\begin{eqnarray*}
\alpha(G)\leq n\frac{|\tau|}{k-\tau}.
\end{eqnarray*}
If an independent set $C$ meets this bound, then every vertex not in $C$ is adjacent to exactly $|\tau|$ vertices of $C$.
\end{thm}
The upper bound in Theorem \ref{thm1.1} is often called Hoffman (ratio) bound for $\alpha(G)$. Haemers generalized the Hoffman bound to the general case.
\begin{thm}\label{thm1.2}
\textup{[17, page 17]} Let $G$ be a graph with minimum degree $\delta$, and let $\lambda$ and $\tau$ be the maximum and the minimum adjacency eigenvalue of $G$, respectively. Then
\begin{eqnarray*}
\alpha(G)\leq n\frac{\lambda|\tau|}{\delta^2-\lambda\tau}.
\end{eqnarray*}
\end{thm}
There are many Hoffman-type ratio bounds, which are useful in graph theory, coding theory and extremal combinatorics \cite{Delsarte,EllisFilmus,Ellis,Frankl,Friedgut}. Some generalized Hoffman ratio bounds involving adjacency eigenvalues, Laplacian eigenvalues and normalized Laplacian eigenvalues of graphs can be found in \cite{Godsil,Haemers,Harant,Lumei}.

For a graph $G=(V,E)$, let $G^k$ denote the graph whose vertex set is $V^k$, in which two vertices $u_1\cdots u_k$ and $v_1\cdots v_k$ are adjacent if and only if for each $i\in\{1,\ldots,k\}$ either $u_i=v_i$ or $\{u_i,v_i\}\in E$ (see \cite{Alon}). The \textit{Shannon capacity} \cite{Brouwer,Shannon} of $G$ is defined as
\begin{eqnarray*}
\Theta(G)=\lim_{k\rightarrow\infty}\alpha(G^k)^{1/k}=\sup_{k\rightarrow\infty}\alpha(G^k)^{1/k},
\end{eqnarray*}
and represents the number of distinct messages per use the channel can communicate with no error while used many times \cite{Alon}. It is difficult to compute $\Theta(G)$ for a general graph $G$, we do not even know the Shannon capacity of the cycle $C_7$ (see \cite{Bohman,Polak}).

The Lov\'{a}sz theta function $\vartheta(G)$ is an upper bound for $\alpha(G)$ introduced in \cite{Lovasz}, which is a powerful tool for studying the Shannon capacity of graphs. Lov\'{a}sz \cite{Lovasz} proved that $\alpha(G)\leq\Theta(G)\leq\vartheta(G)$, and posed the problem of finding graphs with $\Theta(G)=\vartheta(G)$ (see Problem 1 in \cite{Lovasz}). The Schrijver theta function $\vartheta'(G)$ is a smaller upper bound for $\alpha(G)$. Schrijver \cite{Schrijver} proved that $\alpha(G)\leq\vartheta'(G)\leq\vartheta(G)$. By considering the relation between $\alpha(G)$ and the rank of matrices associated with $G$, Haemers \cite{Haemers0} proved that $\Theta(G)$ is at most the minimum rank of a class of graph matrices.

Generalized inverses of graph matrices have important applications in random walks on graphs \cite{ChungYau} and the sparsification of graphs \cite{Spielman}. This paper uses generalized inverses methods to study $\alpha(G)$, $\Theta(G)$, $\vartheta(G)$ and $\vartheta'(G)$, and gives bounds which unify the Lov\'{a}sz theta function, Schrijver theta function and Hoffman-type bounds. Based on the generalized inverses method, we obtain simple structural and spectral conditions to determine $\alpha(G)$, $\Theta(G)$, $\vartheta(G)$ and a maximum independent set in $G$. Our conditions can also be used to find graphs with $\Theta(G)=\vartheta(G)$, which is a partial answer to the problem posed by Lov\'{a}sz.

The paper is organized as follows. In Section 2, we introduce some auxiliary lemmas. In Section 3, we give bounds for the sizes of independent sets and the independence number of graphs by using generalized inverses and eigenvalues of graph matrices. In Sections 4-6, we show that our bounds unify the Lov\'{a}sz theta function, Schrijver theta function and Hoffman-type bounds in \cite{Brouwer,Godsil,Harant,Lumei}, and obtain the necessary and sufficient conditions of graphs attaining the bounds. In Section 7, we give some simple structural and spectral conditions for determining a maximum independent set, the independence number, the Shannon capacity and the Lov\'{a}sz theta function of a graph. Some concluding remarks are given in Section 8.

\section{Preliminaries}
For a real square matrix $M$, the \textit{group inverse} of $M$, denoted by $M^\#$, is the matrix $X$ such that $MXM=M,~XMX=X$ and $MX=XM$. It is known \textup{[2, page 156]} that $M^\#$ exists if and only if $\mbox{\rm rank}(M)=\mbox{\rm rank}(M^2)$. If $M^\#$ exists, then $M^\#$ is unique.

For a positive semidefinite real matrix $M$, there exists an orthogonal matrix $U$ such that
\begin{eqnarray*}
M=U{\rm diag}(\lambda_1,\ldots,\lambda_n)U^\top,
\end{eqnarray*}
where ${\rm diag}(\lambda_1,\ldots,\lambda_n)$ denotes a diagonal matrix with nonnegative diagonal entries $\lambda_1,\ldots,\lambda_n$. We set $M^{\frac{1}{2}}=U{\rm diag}(\lambda_1^{\frac{1}{2}},\ldots,\lambda_n^{\frac{1}{2}})U^\top$ and $M^{-\frac{1}{2}}=U{\rm diag}((\lambda_1^+)^{\frac{1}{2}},\ldots,(\lambda_n^+)^{\frac{1}{2}})U^\top$, where $\lambda_i^+=\lambda_i^{-1}$ if $\lambda_i>0$, and $\lambda_i^+=0$ if $\lambda_i=0$. Then
\begin{eqnarray*}
M^\#=U{\rm diag}(\lambda_1^+,\ldots,\lambda_n^+)U^\top=(M^{-\frac{1}{2}})^2.
\end{eqnarray*}

\begin{lem}\label{lem2.1}
\textup{[2, pages 162-163]} Let $M$ be a square matrix such that $M^\#$ exists, and let $\lambda\neq0$ be an eigenvalue of $M$ with an eigenvector $x$. Then
\begin{eqnarray*}
M^\#x=\lambda^{-1}x.
\end{eqnarray*}
\end{lem}

For a real matrix $A$, the \textit{Moore-Penrose inverse} of $A$ is the real matrix $X$ such that $AXA=A$, $XAX=X$, $(AX)^\top=AX$ and $(XA)^\top=XA$. Let $A^+$ denote the Moore-Penrose inverse of $A$. It is well known that $A^+$ exists and is unique.

Let $R(M)=\{x:x=My,y\in\mathbb{R}^n\}$ denote the range of an $m\times n$ real matrix $M$.
\begin{lem}\label{lem2.2}
\textup{[2, pages 43,49]} Let $A$ be a real matrix. Then
\begin{eqnarray*}
R(AA^\top)&=&R(A)=R(AA^+),\\
(A^\top A)^\#A^\top&=&A^+.
\end{eqnarray*}
\end{lem}

\begin{lem}\label{lem2.4}
Let $A$ be an $n\times m$ real matrix, and let $x$ be a unit real vector of dimension $n$. Then
\begin{eqnarray*}
x^\top AA^+x\leq1,
\end{eqnarray*}
with equality if and only if $x\in R(A)$.
\end{lem}
\begin{proof}
It is known \cite{Ben-Israel} that $AA^+$ is a real symmetric idempotent matrix. So each eigenvalue of $AA^+$ is $0$ or $1$. For a unit real vector $x$, we have $x^\top AA^+x\leq1$, with equality if and only if $x\in R(AA^+)$. By Lemma \ref{lem2.2}, $x\in R(AA^+)$ is equivalent to $x\in R(A)$.
\end{proof}
We now introduce the Lov\'{a}sz theta function defined in \cite{Lovasz}. An \textit{orthonormal representation} of an $n$-vertex graph $G$ is a set $\{u_1,\ldots,u_n\}$ of unit real vectors such that $u_i^\top u_j=0$ if $i$ and $j$ are two nonadjacent vertices in $G$. The \textit{value} of an orthonormal representation $\{u_1,\ldots,u_n\}$ is defined to be
\begin{eqnarray*}
\min_c\max_{1\leq i\leq n}\frac{1}{(c^\top u_i)^2},
\end{eqnarray*}
where $c$ ranges over all unit real vectors. The \textit{Lov\'{a}sz theta function} $\vartheta(G)$ is the minimum value of all orthonormal representations of $G$.

Let $(M)_{ij}$ denote the $(i,j)$-entry of a matrix $M$, and let $\lambda_1(A)$ denote the largest eigenvalue of a real symmetric matrix $A$. The independence number $\alpha(G)$, the Shannon capacity $\Theta(G)$ and the Lov\'{a}sz theta function $\vartheta(G)$ have the following relations.
\begin{lem}\label{lem2.5}
\textup{\cite{Lovasz}} For any graph $G$, we have
\begin{eqnarray*}
\alpha(G)\leq\Theta(G)\leq\vartheta(G)=\min_{A}\lambda_1(A),
\end{eqnarray*}
where $A$ ranges over all real symmetric matrices indexed by vertices of $G$ such that $(A)_{ij}=1$ when $i=j$ or $i,j$ are two nonadjacent vertices.
\end{lem}
For a graph $G$, its \textit{Schrijver theta function} $\vartheta'(G)$ is defined as \cite{Schrijver}
\begin{eqnarray*}
\vartheta'(G)=\min_{A}\lambda_1(A),
\end{eqnarray*}
where $A$ ranges over all real symmetric matrices indexed by vertices of $G$ such that $(A)_{ij}\geq1$ when $i=j$ or $i,j$ are two nonadjacent vertices.
\begin{lem}\label{lem2.7}
\textup{\cite{Schrijver}} For any graph $G$, we have
\begin{eqnarray*}
\alpha(G)\leq\vartheta'(G).
\end{eqnarray*}
\end{lem}
For an $n$-vertex graph $G$, the \textit{adjacency matrix} $A$ of $G$ is an $n\times n$ symmetric matrix with entries
\begin{eqnarray*}
(A)_{ij}=\begin{cases}1~~~~~~~~~~~~~~~~~~\mbox{if}~\{i,j\}\in E(G),\\
0~~~~~~~~~~~~~~~~~~\mbox{if}~\{i,j\}\notin E(G),\end{cases}
\end{eqnarray*}
where $E(G)$ denotes the edge set of $G$. Eigenvalues of the adjacency matrix of $G$ are called \textit{adjacency eigenvalues} of $G$. The largest adjacency eigenvalue of a graph has the following upper bound.

The following lemma will be used later in the proof of Example 4.3.
\begin{lem}\label{lem2.8}
\textup{\cite{Berman}} Let $G$ be a graph with adjacency matrix $A$. Then
\begin{eqnarray*}
\lambda_1(A)\leq\max_{\{u,v\}\in E(G)}\sqrt{d_ud_v},
\end{eqnarray*}
where $d_u$ denotes the degree of a vertex $u$.
\end{lem}

\section{Bounds for independent sets and independence number}
Let $A[S]$ denote the principal submatrix of a square matrix $A$ determined by the rows and columns whose index set is $S$. A real vector $x=(x_1,\ldots,x_n)^\top$ is called \textit{total nonzero} if $x_i\neq0$ for $i=1,\ldots,n$.

For an $n$-vertex graph $G$, let $V(G)$ denote the vertex set of $G$, and let $\mathcal{P}(G)$ denote the set of real matrix-vector pairs $(M,x)$ such that\\
(a) $M$ is a positive semidefinite $n\times n$ matrix indexed by vertices of $G$.\\
(b) $x=(x_1,\ldots,x_n)^\top\in R(M)$ is total nonzero and $(M)_{ij}x_ix_j\leq0$ if $i,j$ are two nonadjacent distinct vertices.

For $(M,x)\in\mathcal{P}(G)$ and vertex subset $T\subseteq V(G)$, let $F(M,x)$ and $F_T(M,x)$ denote the following functions
\begin{eqnarray*}
F(M,x)&=&x^\top M^\#x\max_{u\in V(G)}\frac{(M)_{uu}}{x_u^2},\\
F_T(M,x)&=&\frac{x^\top M^\#x}{|T|}\sum_{u\in T}\frac{(M)_{uu}}{x_u^2}.
\end{eqnarray*}

In \cite{Godsil}, Godsil and Newman gave bounds on the size of an independent set by using positive semidefinite graph matrices. Motivated by the ideas of Godsil and Newman, we give the following bounds in terms of generalized inverses of positive semidefinite graph matrices.
\begin{thm}\label{thm3.1}
Let $(M,x)\in\mathcal{P}(G)$ be a matrix-vector pair associated with an $n$-vertex graph $G$. Then the following statements hold:\\
(1) For any independent set $S$ of $G$, we have
\begin{eqnarray*}
|S|\leq F_S(M,x),
\end{eqnarray*}
with equality if and only if $M[S]$ is diagonal and there exists a constant $c$ such that
\begin{eqnarray*}
\frac{(M)_{vv}}{x_v^2}&=&c~(v\in S),\\
cx_v&=&\sum_{u\in S}(M)_{vu}x_u^{-1}~(v\notin S).
\end{eqnarray*}
(2) The independence number $\alpha(G)$ satisfies
\begin{eqnarray*}
\alpha(G)\leq F(M,x),
\end{eqnarray*}
with equality if and only if $G$ has an independent set $C$ such that $M[C]$ is diagonal and
\begin{eqnarray*}
\frac{(M)_{vv}}{x_v^2}&=&\max_{u\in V(G)}\frac{(M)_{uu}}{x_u^2}=c~(v\in C),\\
cx_v&=&\sum_{u\in C}(M)_{vu}x_u^{-1}~(v\notin C).
\end{eqnarray*}
Moreover, if an independent set $C$ satisfies the above conditions, then
\begin{eqnarray*}
\alpha(G)=|C|=F(M,x).
\end{eqnarray*}
(3) Suppose that $\frac{(M)_{11}}{x_1^2}\geq\cdots\geq\frac{(M)_{nn}}{x_n^2}$ are arranged in a decreasing order. If $\alpha(G)\geq t$ for some positive integer $t$, then
\begin{eqnarray*}
\alpha(G)\leq F_T(M,x),
\end{eqnarray*}
where $T=\{1,\ldots,t\}$.
\end{thm}
\begin{proof}
Since $M^{\frac{1}{2}}M^{-\frac{1}{2}}M=M$ and $x\in R(M)$, we have
\begin{eqnarray*}
M^{\frac{1}{2}}M^{-\frac{1}{2}}x=x.
\end{eqnarray*}
For any independent set $S$ of $G$, let $y=(y_1,\ldots,y_n)^\top$ be the vector such that
\begin{eqnarray*}
y_u=\begin{cases}x_u^{-1}~~~~~~~~~~~~~\mbox{if}~u\in S,\\
0~~~~~~~~~~~~~~~~\mbox{if}~u\notin S.\end{cases}
\end{eqnarray*}
By the Cauchy-Schwarz inequality, we have
\begin{eqnarray*}
|S|^2=(y^\top x)^2=(y^\top M^{\frac{1}{2}}M^{-\frac{1}{2}}x)^2\leq(y^\top My)(x^\top M^\#x).
\end{eqnarray*}
Since $(M)_{ij}x_ix_j\leq0$ when $i,j$ are two nonadjacent vertices, we get
\begin{eqnarray*}
|S|^2&\leq&(y^\top My)(x^\top M^\#x)\leq x^\top M^\#x\sum_{u\in S}\frac{(M)_{uu}}{x_u^2},\\
|S|&\leq&F_S(M,x),
\end{eqnarray*}
with equality if and only if $M[S]$ is diagonal and there exists a constant $c$ such that $M^{\frac{1}{2}}y=cM^{-\frac{1}{2}}x$, that is, $My=cx$ (because $M^{\frac{1}{2}}M^{-\frac{1}{2}}x=x$). When $M[S]$ is diagonal, $My=cx$ is equivalent to
\begin{eqnarray*}
\frac{(M)_{vv}}{x_v^2}&=&c~(v\in S),\\
cx_v&=&\sum_{u\in S}(M)_{vu}x_u^{-1}~(v\notin S).
\end{eqnarray*}
So part (1) holds.

Notice that $|S|\leq F_S(M,x)\leq F(M,x)$ for any independent set $S$ of $G$. Hence
\begin{eqnarray*}
\alpha(G)\leq F(M,x).
\end{eqnarray*}
If the equality holds, then $G$ has an independent set $C$ such that $M[C]$ is diagonal and
\begin{eqnarray*}
\frac{(M)_{vv}}{x_v^2}&=&\max_{u\in V(G)}\frac{(M)_{uu}}{x_u^2}=c~(v\in C),\\
cx_v&=&\sum_{u\in C}(M)_{vu}x_u^{-1}~(v\notin C).
\end{eqnarray*}
If an independent set $C$ satisfies the above conditions, then
\begin{eqnarray*}
|C|^2&=&(z^\top x)^2=(z^\top M^{\frac{1}{2}}M^{-\frac{1}{2}}x)^2=(z^\top Mz)(x^\top M^\#x)=|C|F(M,x),\\
|C|&=&F(M,x),
\end{eqnarray*}
where $z=(z_1,\ldots,z_n)^\top$ is the vector such that
\begin{eqnarray*}
z_u=\begin{cases}x_u^{-1}~~~~~~~~~~~~~\mbox{if}~u\in C,\\
0~~~~~~~~~~~~~~~~\mbox{if}~u\notin C.\end{cases}
\end{eqnarray*}
Since $\alpha(G)\leq F(M,x)$, we have
\begin{eqnarray*}
|C|=\alpha(G)=F(M,x).
\end{eqnarray*}
So part (2) holds.

Suppose that $\frac{(M)_{11}}{x_1^2}\geq\cdots\geq\frac{(M)_{nn}}{x_n^2}$ are arranged in a decreasing order. Let $C$ be an independent set such that $|C|=\alpha(G)$. If $\alpha(G)\geq t$, then
\begin{eqnarray*}
\alpha(G)\leq F_C(M,x)\leq F_T(M,x)~(T=\{1,\ldots,t\}).
\end{eqnarray*}
So part (3) holds.
\end{proof}

\begin{rem}
For a graph $G$, let $M$ be a positive definite matrix indexed by vertices of $G$ such that $(M)_{ij}=0$ if $i,j$ are two nonadjacent distinct vertices. Then $M^\#=M^{-1}$ and $(M,x)\in\mathcal{P}(G)$ for any real vector $x$. By part (2) of Theorem \ref{thm3.1}, we have
\begin{eqnarray*}
\alpha(G)\leq F(M,x)=x^\top M^{-1}x\max_{u\in V(G)}\frac{(M)_{uu}}{x_u^2}
\end{eqnarray*}
for any total nonzero vector $x$.
\end{rem}
The \textit{$\{1\}$-inverse} of $M$ is a matrix $X$ such that $MXM=M$. If $M$ is singular, then it has infinite $\{1\}$-inverses (see \textup{[2, page 41]}).
\begin{rem}
Let $M$ be a real symmetric matrix. Since $MM^\#M=MM^{(1)}M=M$ for any $\{1\}$-inverse $M^{(1)}$ of $M$, we have $x^\top M^\#x=x^\top M^{(1)}x$ for any $x\in R(M)$. Hence the upper bounds given in Theorem \ref{thm3.1} can be obtained from any $\{1\}$-inverse of $M$.
\end{rem}
The \textit{signless Laplacian matrix} of a graph $G$ is defined as $D+A$, where $D$ is the diagonal matrix of vertex degrees of $G$, $A$ is the adjacency matrix of $G$. The \textit{subdivision graph} of $G$, denoted by $S(G)$, is a graph obtained from $G$ by replacing each edge of $G$ by a path of length $2$.

We can find concrete examples such that an upper bound derived from Theorem \ref{thm3.1} is smaller than the Hoffman bound given in Theorem \ref{thm1.2}.
\begin{example}
Let $G$ be a nonregular graph with $n$ vertices, $m$ edges and minimum degree $\delta\geq2$. Then there exists $(M,x)\in\mathcal{P}(S(G))$ such that
\begin{eqnarray*}
\alpha(S(G))=x^\top M^\#x\max_{u\in V(S(G))}\frac{(M)_{uu}}{x_u^2}<\frac{(m+n)\lambda^2}{4+\lambda^2},
\end{eqnarray*}
where $\lambda$ is the maximum adjacency eigenvalue of $S(G)$.
\end{example}
\begin{proof}
The signless Laplacian matrix of $S(G)$ can be written as $M=\begin{pmatrix}2I&B^\top\\B&D\end{pmatrix}$, where $B$ is the vertex-edge incidence matrix of $G$, $D$ is the diagonal matrix of vertex degrees of $G$. Let $x=\begin{pmatrix}2I&B^\top\\B&D\end{pmatrix}\begin{pmatrix}e\\0\end{pmatrix}=\begin{pmatrix}2e\\y\end{pmatrix}$, where $e$ is the all-ones vector, $y$ is the vector of vertex degrees of $G$. Let $C$ be the independent set in $S(G)$ corresponding to the edge set of $G$, then $(M,x)\in\mathcal{P}(S(G))$ and $C$ satisfy the conditions given in part (2) of Theorem \ref{thm3.1}. By part (2) of Theorem \ref{thm3.1}, we get
\begin{eqnarray*}
\alpha(S(G))=x^\top M^\#x\max_{u\in V(S(G))}\frac{(M)_{uu}}{x_u^2}=|C|=m.
\end{eqnarray*}
Let $\lambda$ be the maximum adjacency eigenvalue of $S(G)$, then $-\lambda$ is the minimum adjacency eigenvalue of $S(G)$ (because $S(G)$ is bipartite). By Theorem \ref{thm1.2}, we have
\begin{eqnarray*}
\alpha(S(G))=x^\top M^\#x\max_{u\in V(S(G))}\frac{(M)_{uu}}{x_u^2}=m\leq\frac{(m+n)\lambda^2}{4+\lambda^2}.
\end{eqnarray*}
We next show that the inequality is strict. It is known that $\lambda^2$ is the maximum eigenvalue of the signless Laplacian matrix of $G$, and $\lambda^2\geq\frac{4m}{n}$, with equality if and only if $G$ is regular (see [7, Theorems 2.4.4 and 7.8.6]). Since $G$ is nonregular, we have $\lambda^2>\frac{4m}{n}$, that is $m<\frac{(m+n)\lambda^2}{4+\lambda^2}$.
\end{proof}

For a connected graph $G$, if $M$ be a nonnegative matrix indexed by $V(G)$ such that $(M)_{ij}>0$ if and only if $i,j$ are two adjacent vertices, then $M$ is irreducible. By the Perron-Frobenius Theorem, the spectral radius $\rho$ of $M$ is a positive eigenvalue of $M$, and there exists unique positive unit eigenvector associated with $\rho$. The following eigenvalue bound follows from Theorem \ref{thm3.1}.
\begin{cor}
Let $M$ be a positive semidefinite nonnegative matrix associated with a connected graph $G$, and $(M)_{ij}>0$ if and only if $i,j$ are two adjacent vertices. Let $\rho$ be the spectral radius of $M$, and let $x$ be the corresponding positive unit eigenvector. Then
\begin{eqnarray*}
\alpha(G)\leq\rho^{-1}\max_{u\in V(G)}\frac{(M)_{uu}}{x_u^2}.
\end{eqnarray*}
\end{cor}
\begin{proof}
Since $Mx=\rho x$, we have $(M,x)\in\mathcal{P}(G)$. By Lemma \ref{lem2.1}, we get
\begin{eqnarray*}
M^\#x=\rho^{-1}x.
\end{eqnarray*}
By part (2) of Theorem \ref{thm3.1}, we get
\begin{eqnarray*}
\alpha(G)\leq F(M,x)=x^\top M^\#x\max_{u\in V(G)}\frac{(M)_{uu}}{x_u^2}=\rho^{-1}\max_{u\in V(G)}\frac{(M)_{uu}}{x_u^2}.
\end{eqnarray*}
\end{proof}
The \textit{Laplacian matrix} of a graph $G$ is defined as $L=D-A$, where $D$ is the diagonal matrix of vertex degrees of $G$, $A$ is the adjacency matrix of $G$. It is well known that $L$ is positive semidefinite. The following result follows from Theorem \ref{thm3.1}.
\begin{cor}
Let $G$ be a connected graph with Laplacian matrix $L$. For any total nonzero vector $x$ satisfying $\sum_{u\in V(G)}x_u=0$, we have
\begin{eqnarray*}
\alpha(G)\leq F(L,x).
\end{eqnarray*}
\end{cor}
\begin{proof}
Since $G$ is connected, the null space of $L$ has dimension one. Since $Le=0$ for the all-ones vector $e$, we have $x\in R(L)$ for any vector $x$ satisfying $e^\top x=\sum_{u\in V(G)}x_u=0$. By part (2) of Theorem \ref{thm3.1}, we get
\begin{eqnarray*}
\alpha(G)\leq F(L,x).
\end{eqnarray*}
\end{proof}

\section{Lov\'{a}sz theta function and Shannon capacity}
For a graph $G$, let $\mathcal{M}(G)$ denote the set of real matrix-vector pairs $(M,x)$ such that\\
(a) $M$ is a positive semidefinite matrix indexed by vertices of $G$ such that $(M)_{ij}=0$ if $i,j$ are two nonadjacent distinct vertices.\\
(b) $x\in R(M)$ is a total nonzero vector.

Since $\mathcal{M}(G)$ is a subset of $\mathcal{P}(G)$, the notations $F(M,x)$ and $F_T(M,x)$ defined in Section 3 can be used for $(M,x)\in\mathcal{M}(G)$ directly.

In the following theorem, we show that the Lov\'{a}sz theta function $\vartheta(G)$ is a special upper bound given in Theorem \ref{thm3.1}, and obtain the necessary and sufficient conditions of graphs satisfying $\alpha(G)=\vartheta(G)$.
\begin{thm}\label{thm4.1}
For any graph $G$, we have
\begin{eqnarray*}
\alpha(G)\leq\min_{(M,x)\in\mathcal{M}(G)}F(M,x)=\vartheta(G),
\end{eqnarray*}
with equality if and only if there exist $(M,x)\in\mathcal{M}(G)$ and an independent set $C$ satisfying the conditions given in part (2) of Theorem \ref{thm3.1}.
\end{thm}
\begin{proof}
By Theorem \ref{thm3.1}, we have
\begin{eqnarray*}
\alpha(G)\leq\min_{(M,x)\in\mathcal{M}(G)}F(M,x),
\end{eqnarray*}
with equality if and only if there exist $(M,x)\in\mathcal{M}(G)$ and an independent set $C$ satisfying the conditions in part (2) of Theorem \ref{thm3.1}. We need to prove that the minimum value equals to $\vartheta(G)$.

We first prove that the minimum value does not exceed $\vartheta(G)$. Let $c,u_1,\ldots,u_n$ be unit real vectors such that $\{u_1,\ldots,u_n\}$ is an orthonormal representation of $G$ and
\begin{eqnarray*}
\vartheta(G)=\max_{1\leq i\leq n}\frac{1}{(c^\top u_i)^2}.
\end{eqnarray*}
Let $U$ be the matrix whose $i$-th column is $u_i$, then each diagonal entry of $U^\top U$ is $1$. Take $z=U^\top c$, then $(U^\top U,z)\in\mathcal{M}(G)$ and
\begin{eqnarray*}
z_i=c^\top u_i\neq0,~i=1,\ldots,n.
\end{eqnarray*}
By Lemmas \ref{lem2.2} and \ref{lem2.4}, we have
\begin{eqnarray*}
F(U^\top U,z)=z^\top(U^\top U)^\#z\max_{1\leq i\leq n}\frac{1}{z_i^2}=\vartheta(G)c^\top UU^+c\leq\vartheta(G).
\end{eqnarray*}
So we have
\begin{eqnarray*}
\min_{(M,x)\in\mathcal{M}(G)}F(M,x)\leq F(U^\top U,z)\leq\vartheta(G).
\end{eqnarray*}

We turn to prove that $\min_{(M,x)\in\mathcal{M}(G)}F(M,x)$ is at least $\vartheta(G)$. There exists $(M_0,y)\in\mathcal{M}(G)$ such that
\begin{eqnarray*}
F(M_0,y)=y^\top(M_0)^\#y\max_{1\leq i\leq n}\frac{(M_0)_{ii}}{y_i^2}=\min_{(M,x)\in\mathcal{M}(G)}F(M,x).
\end{eqnarray*}
Since $M_0$ is positive semidefinite, there exists real matrix $B$ such that $M_0=B^\top B$. Since $y\in R(M_0)$, we can choose $y$ such that $y=B^\top c$ for a unit real vector $c\in R(B)$. Let $b_i$ be the $i$-th column of $B$. By Lemmas \ref{lem2.2} and \ref{lem2.4}, we have
\begin{eqnarray*}
F(M_0,y)=y^\top(M_0)^\#y\max_{1\leq i\leq n}\frac{(M_0)_{ii}}{y_i^2}=c^\top BB^+c\max_{1\leq i\leq n}\frac{b_i^\top b_i}{(c^\top b_i)^2}=\max_{1\leq i\leq n}\frac{b_i^\top b_i}{(c^\top b_i)^2}.
\end{eqnarray*}
We next prove that $F(M_0,y)\geq\vartheta(G)$ by using the method similar with proof of [19, Thereom 3]. Let $A=(a_{ij})_{n\times n}$ be the matrix with entries
\begin{eqnarray*}
a_{ii}&=&1,~i=1,\ldots,n,\\
a_{ij}&=&1-\frac{b_i^\top b_j}{(c^\top b_i)(c^\top b_j)}=1-\frac{(M_0)_{ij}}{y_iy_j},~i\neq j.
\end{eqnarray*}
By Lemma \ref{lem2.5}, we get
\begin{eqnarray*}
\lambda_1(A)\geq\vartheta(G).
\end{eqnarray*}
We set
\begin{eqnarray*}
\beta=F(M_0,y)=\max_{1\leq i\leq n}\frac{b_i^\top b_i}{(c^\top b_i)^2},
\end{eqnarray*}
then
\begin{eqnarray*}
-a_{ij}&=&\left(c-\frac{b_i}{(c^\top b_i)}\right)^\top\left(c-\frac{b_j}{(c^\top b_j)}\right),~i\neq j,\\
\beta-a_{ii}&=&\left(c-\frac{b_i}{(c^\top b_i)}\right)^\top\left(c-\frac{b_i}{(c^\top b_i)}\right)+\beta-\frac{b_i^\top b_i}{(c^\top b_i)^2}.
\end{eqnarray*}
Hence $\beta I-A$ ($I$ is the identity matrix) is positive semidefinite, that is
\begin{eqnarray*}
\min_{(M,x)\in\mathcal{M}(G)}F(M,x)=\beta\geq\lambda_1(A)\geq\vartheta(G).
\end{eqnarray*}
\end{proof}
The following conclusion comes from Theorem \ref{thm3.1}, Theorem \ref{thm4.1} and Lemma \ref{lem2.5}.
\begin{thm}\label{thm4.2}
Let $(M,x)\in\mathcal{M}(G)$ be a matrix-vector pair for graph $G$. If $G$ has an independent set $C$ such that $(M,x)$ and $C$ satisfy the conditions given in part (2) of Theorem \ref{thm3.1}, then
\begin{eqnarray*}
\alpha(G)=\Theta(G)=\vartheta(G)=|C|=F(M,x).
\end{eqnarray*}
\end{thm}

A bipartite graph $G$ is called \textit{semiregular} with parameters $(n_1,n_2,r_1,r_2)$ if $V(G)$ has a bipartition $V(G)=V_1\cup V_2$ such that $|V_1|=n_1,|V_2|=n_2$ and vertices in the same colour class have the same degree ($n_i$ vertices of degree $r_i$, $i=1,2$).

Let $I$ denote the identity matrix. The following are some examples for Theorem \ref{thm4.2}.
\begin{example}
Let $G$ be a graph with a spanning subgraph $H$, and let $H$ be a semiregular bipartite graph with parameters $(n_1,n_2,r_1,r_2)$ ($n_1r_1=n_2r_2>0,n_1\leq n_2$). If the independent set with $n_2$ vertices in $H$ is also an independent set in $G$, then
\begin{eqnarray*}
\alpha(G)=\Theta(G)=\vartheta(G)=n_2.
\end{eqnarray*}
\end{example}
\begin{proof}
The adjacency matrix of $H$ can be written as $A=\begin{pmatrix}0&B^\top\\B&0\end{pmatrix}$, and each row (column) sum of $B\in\mathbb{R}^{n_1\times n_2}$ is $r_1$ ($r_2$). By Lemma \ref{lem2.8}, we have $\lambda_1(A)\leq\sqrt{r_1r_2}$. Then
\begin{eqnarray*}
\lambda_1(BB^\top)=\lambda_1(A^2)\leq r_1r_2.
\end{eqnarray*}
Let $M=\begin{pmatrix}r_1I&B^\top\\B&r_2I\end{pmatrix}$, then the all-ones vector $e\in R(M)$ because $Me=(r_1+r_2)e$. Since the Schur complement
\begin{eqnarray*}
r_2I-r_1^{-1}BB^\top=r_1^{-1}(r_1r_2I-BB^\top)
\end{eqnarray*}
is positive semidefinite, $M$ is positive semidefinite. Hence $(M,e)\in\mathcal{M}(G)$. Let $C$ be the independent set with $n_2$ vertices in $H$, which is also an independent set in $G$. Since $(M,e)$ and $C$ satisfy the conditions given in part (2) of Theorem \ref{thm3.1}, then by Theorem \ref{thm4.2}, we have
\begin{eqnarray*}
\alpha(G)=\Theta(G)=\vartheta(G)=|C|=\alpha(H).
\end{eqnarray*}
\end{proof}
For two disjoint graphs $G_1$ and $G_2$, let $G_1\vee G_2$ denote the graph obtained from the union of $G_1$ and $G_2$ by joining each vertex of $G_1$ to each vertex of $G_2$. Let $\overline{G}$ denote the complement of a graph $G$.
\begin{example}
Let $G$ be an $n$-vertex graph with minimum degree $\delta$. For any integer $s\geq n-\delta$, we have
\begin{eqnarray*}
\alpha(G\vee\overline{K_s})=\Theta(G\vee\overline{K_s})=\vartheta(G\vee\overline{K_s})=s,
\end{eqnarray*}
where $K_s$ is the complete graph with $s$ vertices.
\end{example}
\begin{proof}
The adjacency matrix of $G\vee\overline{K_s}$ can be written as $\begin{pmatrix}0&J_{s\times n}\\J_{s\times n}^\top&A\end{pmatrix}$, where $A$ is the adjacency matrix of $G$, $J_{s\times n}$ is the $s\times n$ all-ones matrix. Let $L$ be the Laplacian matrix of $G$, and let $M=\begin{pmatrix}sI&J_{s\times n}\\J_{s\times n}^\top&nI-L\end{pmatrix}$. Then the all-ones vector $e\in R(M)$ because $Me=(s+n)e$. Since the Schur complement
\begin{eqnarray*}
nI-L-s^{-1}J_{s\times n}^\top J_{s\times n}
\end{eqnarray*}
equals to the Laplacian matrix of $\overline{G}$ (which is positive semidefinite), $M$ is positive semidefinite. Hence $(M,e)\in\mathcal{M}(G)$. Let $C$ be the independent set of size $s$ in $G\vee\overline{K_s}$. Since $(M,e)$ and $C$ satisfy the conditions given in part (2) of Theorem \ref{thm3.1}, then by Theorem \ref{thm4.2}, we have
\begin{eqnarray*}
\alpha(G\vee\overline{K_s})=\Theta(G\vee\overline{K_s})=\vartheta(G\vee\overline{K_s})=|C|=s.
\end{eqnarray*}
\end{proof}
Let $e=(1,\ldots,1)^\top$ denote the all-ones vector.
\begin{cor}\label{thm4.5}
Let $G$ be a graph with adjacency matrix $A$ and $A+\lambda I$ is positive definite for some $\lambda>0$. Then
\begin{eqnarray*}
\alpha(G)\leq\lambda e^\top(A+\lambda I)^{-1}e,
\end{eqnarray*}
with equality if and only if there exists an independent set $C$ such that every vertex not in $C$ is adjacent to exactly $\lambda$ vertices of $C$. Moreover, if an independent set $C$ satisfies the above condition, then
\begin{eqnarray*}
\alpha(G)=\Theta(G)=\vartheta(G)=|C|=\lambda e^\top(A+\lambda I)^{-1}e.
\end{eqnarray*}
\end{cor}
\begin{proof}
Since $A+\lambda I$ is positive definite, we have $(A+\lambda I,e)\in\mathcal{M}(G)$ and $F(A+\lambda I,e)=\lambda e^\top(A+\lambda I)^{-1}e$. The conclusions follow from Theorems \ref{thm3.1} and \ref{thm4.2}.
\end{proof}
The following is an example for Corollary \ref{thm4.5}.
\begin{example}
For the path $P_{2k+1}$ with $2k+1$ vertices, let $C$ be the independent set in $P_{2k+1}$ with $k+1$ vertices. Then every vertex not in $C$ is adjacent to exactly $2$ vertices of $C$. By Corollary \ref{thm4.5}, we have
\begin{eqnarray*}
\alpha(P_{2k+1})=\Theta(P_{2k+1})=\vartheta(P_{2k+1})=|C|=2e^\top(A+2I)^{-1}e=k+1,
\end{eqnarray*}
where $A$ is the adjacency matrix of $P_{2k+1}$.
\end{example}
A \textit{clique covering} of $G$ is a set $\varepsilon=\{Q_1,\ldots,Q_r\}$ of cliques in $G$ that cover all edges of $G$, that is, each edge of $G$ belongs to at least one $Q_i$. For a clique covering $\varepsilon$ and a vertex $u$ of $G$, the \textit{$\varepsilon$-degree} of $u$ is the number of cliques in $\varepsilon$ containing $u$, denoted by $d_u^\varepsilon$. The \textit{clique covering matrix} $A_\varepsilon$ is the $n\times n$ symmetric matrix with entries
\begin{eqnarray*}
(A_\varepsilon)_{ij}=\begin{cases}d_i^\varepsilon~~~~~~~~~~~~~~\mbox{if}~i=j,\\
w_{ij}~~~~~~~~~~~~\mbox{if}~\{i,j\}\in E(G),\\
0~~~~~~~~~~~~~~~\mbox{if}~\{i,j\}\notin E(G),\end{cases}
\end{eqnarray*}
where $w_{ij}$ is the number of cliques in $\varepsilon$ containing the edge $\{i,j\}$.
\begin{thm}\label{thm4.7}
Let $G$ be a graph without isolated vertices. For any clique covering $\varepsilon$ of $G$, we have
\begin{eqnarray*}
\alpha(G)\leq F(A_\varepsilon,x),
\end{eqnarray*}
where $x$ is a vector satisfying $x_u=d_u^\varepsilon$. The equality holds if and only if there exists an independent set $C$ such that $d_v^\varepsilon=\min_{u\in V(G)}d_u^\varepsilon$ for each $v\in C$, and $|C\cap Q|=1$ for each $Q\in\varepsilon$. Moreover, if an independent set $C$ satisfies the above conditions, then
\begin{eqnarray*}
\alpha(G)=\Theta(G)=\vartheta(G)=|C|=F(A_\varepsilon,x).
\end{eqnarray*}
\end{thm}
\begin{proof}
For a clique covering $\varepsilon=\{Q_1,\ldots,Q_r\}$ of $G$, the corresponding vertex-clique incidence matrix $B$ is a $|V(G)|\times r$ matrix with entries
\begin{eqnarray*}
(B)_{ij}=\begin{cases}1~~~~~~~~~~~~~~~~~~\mbox{if}~i\in V(G),i\in Q_j,\\
0~~~~~~~~~~~~~~~~~~\mbox{if}~i\in V(G),i\notin Q_j.\end{cases}
\end{eqnarray*}
Then $BB^\top=A_\varepsilon$. Notice that $x=Be\in R(B)=R(A_\varepsilon)$, so $(A_\varepsilon,x)\in\mathcal{M}(G)$. The conclusions follow from Theorems \ref{thm3.1} and \ref{thm4.2}.
\end{proof}
Let $m(H)$ denote the matching number of a hypergraph $H$. The \textit{intersection graph} $\Omega(H)$ of $H$ has vertex set $E(H)$, and two vertices $f_1,f_2$ of $\Omega(H)$ are adjacent if and only if $f_1\cap f_2\neq\emptyset$. Clearly, we have $m(H)=\alpha(\Omega(H))$. The vertex-edge incidence matrix $B$ of $H$ is a $|V(H)|\times|E(H)|$ matrix with entries
\begin{eqnarray*}
(B)_{uf}=\begin{cases}1~~~~~~~~~~~~~~~~~~\mbox{if}~u\in V(H),u\in f\in E(H),\\
0~~~~~~~~~~~~~~~~~~\mbox{if}~u\in V(H),u\notin f\in E(H).\end{cases}
\end{eqnarray*}
We say that $H$ is \textit{$k$-uniform} if each edge of $H$ has exactly $k$ vertices.
\begin{thm}\label{thm4.9}
Let $H$ be a $k$-uniform hypergraph without isolated vertices, and let $B$ be the vertex-edge incidence matrix of $H$. Then
\begin{eqnarray*}
m(H)\leq ke^\top(B^\top B)^\#e,
\end{eqnarray*}
with equality if and only if $H$ has a perfect matching. If $H$ has a perfect matching, then
\begin{eqnarray*}
m(H)=\Theta(\Omega(H))=\vartheta(\Omega(H))=ke^\top(B^\top B)^\#e.
\end{eqnarray*}
\end{thm}
\begin{proof}
Notice that $e=k^{-1}B^\top e\in R(B^\top)=R(B^\top B)$, so $(B^\top B,e)\in\mathcal{M}(\Omega(H))$ and $F(B^\top B,e)=ke^\top(B^\top B)^\#e$. The conclusions follow from Theorems \ref{thm3.1} and \ref{thm4.2}.
\end{proof}

\section{Schrijver theta function}
In the following theorem, we show that the Schrijver theta function $\vartheta(G)$ is the minimum upper bound given in part (2) of Theorem \ref{thm3.1}, and obtain the necessary and sufficient conditions of graphs satisfying $\alpha(G)=\vartheta'(G)$.
\begin{thm}\label{thm5.1}
For any graph $G$, we have
\begin{eqnarray*}
\alpha(G)\leq\min_{(M,x)\in\mathcal{P}(G)}F(M,x)=\vartheta'(G),
\end{eqnarray*}
with equality if and only if there exist $(M,x)\in\mathcal{P}(G)$ and an independent set $C$ satisfying the conditions given in part (2) of Theorem \ref{thm3.1}.
\end{thm}
\begin{proof}
By Theorem \ref{thm3.1}, we have
\begin{eqnarray*}
\alpha(G)\leq\min_{(M,x)\in\mathcal{P}(G)}F(M,x),
\end{eqnarray*}
with equality if and only if there exist $(M,x)\in\mathcal{P}(G)$ and an independent set $C$ satisfying the conditions given in part (2) of Theorem \ref{thm3.1}. We will prove that the minimum value equals to $\vartheta'(G)$ by using the method similar with proof of [19, Thereom 3].

We first prove that the minimum value does not exceed $\vartheta'(G)$. Let $A=(a_{ij})_{n\times n}$ be the real symmetric matrix indexed by vertices of $G$ such that $\lambda_1(A)=\vartheta'(G)$, and $a_{ij}\geq1$ when $i=j$ or $i,j$ are two nonadjacent vertices. Then $\vartheta'(G)I-A$ is positive semidefinite. So there exists real matrix $T$ such that
\begin{eqnarray*}
\vartheta'(G)I-A=T^\top T
\end{eqnarray*}
and the rank of $T$ is smaller than the number of rows of $T$. Let $c$ be a unit real vector such that $c^\top T=0$. Suppose that $t_i$ is the $i$-th column of $T$, then
\begin{eqnarray*}
t_i^\top t_j=\vartheta'(G)\delta_{ij}-a_{ij},
\end{eqnarray*}
where $\delta_{ij}=1$ if $i=j$, and $\delta_{ij}=0$ if $i\neq j$. We set
\begin{eqnarray*}
u_i=\vartheta'(G)^{-\frac{1}{2}}(c+t_i),~i=1,\ldots,n,
\end{eqnarray*}
and let $U$ be the matrix whose $i$-th column is $u_i$. Then
\begin{eqnarray*}
(U^\top U)_{ii}&=&\vartheta'(G)^{-1}(1+\vartheta'(G)-a_{ii})\leq1,~i\in V(G),\\
(U^\top U)_{ij}&=&\vartheta'(G)^{-1}(1-a_{ij})\leq0,~\{i,j\}\notin E(G).
\end{eqnarray*}
Take $z=U^\top c$, then $z_i=c^\top u_i=\vartheta'(G)^{-\frac{1}{2}}$ ($i=1,\ldots,n$) and $(U^\top U,z)\in\mathcal{P}(G)$. By Lemmas \ref{lem2.2} and \ref{lem2.4}, we have
\begin{eqnarray*}
F(U^\top U,z)=z^\top(U^\top U)^\#z\max_{1\leq i\leq n}\frac{(U^\top U)_{ii}}{z_i^2}\leq\vartheta'(G)c^\top UU^+c\leq\vartheta'(G).
\end{eqnarray*}
So we have
\begin{eqnarray*}
\min_{(M,x)\in\mathcal{P}(G)}F(M,x)\leq F(U^\top U,z)\leq\vartheta'(G).
\end{eqnarray*}

We next prove that $\min_{(M,x)\in\mathcal{P}(G)}F(M,x)$ is at least $\vartheta'(G)$. There exists $(M_0,y)\in\mathcal{P}(G)$ such that
\begin{eqnarray*}
F(M_0,y)=y^\top(M_0)^\#y\max_{1\leq i\leq n}\frac{(M_0)_{ii}}{y_i^2}=\min_{(M,x)\in\mathcal{P}(G)}F(M,x).
\end{eqnarray*}
Since $M_0$ is positive semidefinite, there exists matrix $B$ such that $M_0=B^\top B$. Since $y\in R(M_0)$, we can choose $y$ such that $y=B^\top c$ for a unit real vector $c\in R(B)$. Let $b_i$ be the $i$-th column of $B$. By Lemmas \ref{lem2.2} and \ref{lem2.4}, we have
\begin{eqnarray*}
F(M_0,y)=y^\top(M_0)^\#y\max_{1\leq i\leq n}\frac{(M_0)_{ii}}{y_i^2}=c^\top BB^+c\max_{1\leq i\leq n}\frac{b_i^\top b_i}{(c^\top b_i)^2}=\max_{1\leq i\leq n}\frac{b_i^\top b_i}{(c^\top b_i)^2}.
\end{eqnarray*}
Let $A=(a_{ij})_{n\times n}$ be the matrix with entries
\begin{eqnarray*}
a_{ii}&=&1,~i=1,\ldots,n,\\
a_{ij}&=&1-\frac{b_i^\top b_j}{(c^\top b_i)(c^\top b_j)}=1-\frac{(M_0)_{ij}}{y_iy_j},~i\neq j.
\end{eqnarray*}
Since $(M_0)_{ij}y_iy_j\leq0$ for nonadjacent vertices $i,j$, we have $a_{ij}\geq1$ when $i=j$ or $i,j$ are two nonadjacent vertices. By the definition of $\vartheta'(G)$, we have
\begin{eqnarray*}
\lambda_1(A)\geq\vartheta'(G).
\end{eqnarray*}
We set
\begin{eqnarray*}
\beta=F(M_0,y)=\max_{1\leq i\leq n}\frac{b_i^\top b_i}{(c^\top b_i)^2},
\end{eqnarray*}
then
\begin{eqnarray*}
-a_{ij}&=&\left(c-\frac{b_i}{(c^\top b_i)}\right)^\top\left(c-\frac{b_j}{(c^\top b_j)}\right),~i\neq j,\\
\beta-a_{ii}&=&\left(c-\frac{b_i}{(c^\top b_i)}\right)^\top\left(c-\frac{b_i}{(c^\top b_i)}\right)+\beta-\frac{b_i^\top b_i}{(c^\top b_i)^2}.
\end{eqnarray*}
Hence $\beta I-A$ is positive semidefinite, that is
\begin{eqnarray*}
\min_{(M,x)\in\mathcal{P}(G)}F(M,x)=\beta\geq\lambda_1(A)\geq\vartheta'(G).
\end{eqnarray*}
\end{proof}
The following is an example for a graph $G$ satisfying $\alpha(G)=\vartheta'(G)<\vartheta(G)$ (see \cite{Schrijver}).
\begin{example}
\textup{\rm \cite{Schrijver}} Let $G$ be a graph with vertex set $\{0,1\}^6$, and two vertices are adjacent if and only if their Hamming distance is at most $3$. Then $\{000000,001111,110011,111100\}$ is an independent set of $G$. It is known that $4=\vartheta'(G)<\vartheta(G)=16/3$.
\end{example}
By Theorem \ref{thm3.1}, Theorem \ref{thm5.1} and Lemma \ref{lem2.7}, we can get the following conclusion.
\begin{thm}\label{thm5.3}
Let $(M,x)\in\mathcal{P}(G)$ be a matrix-vector pair for graph $G$. If $G$ has an independent set $C$ such that $(M,x)$ and $C$ satisfy the conditions given in part (2) of Theorem \ref{thm3.1}, then
\begin{eqnarray*}
\alpha(G)=\vartheta'(G)=|C|=F(M,x).
\end{eqnarray*}
\end{thm}

\section{Hoffman-type bounds}
In this section, we show that the Hoffman-type bounds in \cite{Brouwer,Godsil,Harant,Lumei} are special cases of Theorem \ref{thm3.1}.

By Theorem \ref{thm3.1}, we can get the following classical Hoffman bound, and obtain the necessary and sufficient conditions of graphs attaining the Hoffman bound. The necessary condition of the equality case given in Theorem \ref{thm1.1} is indeed a sufficient condition.
\begin{cor}\label{thm6.1}
Let $G$ be a $k$-regular ($k\neq0$) graph with $n$ vertices, and let $\tau$ be the minimum adjacency eigenvalue of $G$. Then
\begin{eqnarray*}
\alpha(G)\leq n\frac{|\tau|}{k-\tau},
\end{eqnarray*}
with equality if and only if there exists an independent set $C$ such that every vertex not in $C$ is adjacent to exactly $|\tau|$ vertices of $C$. Moreover, if an independent set $C$ satisfies the above condition, then
\begin{eqnarray*}
\alpha(G)=\Theta(G)=\vartheta(G)=|C|=n\frac{|\tau|}{k-\tau}.
\end{eqnarray*}
\end{cor}
\begin{proof}
Since $G$ is $k$-regular, we have $(A-\tau I,e)\in\mathcal{M}(G)$, where $A$ is the adjacency matrix of $G$. By Lemma \ref{lem2.1}, we get
\begin{eqnarray*}
(A-\tau I)^\#e&=&(k-\tau)^{-1}e,\\
F(A-\tau I,e)&=&|\tau|e^\top(A-\tau I)^\#e=n\frac{|\tau|}{k-\tau}.
\end{eqnarray*}
The conclusions follow from Theorems \ref{thm3.1} and \ref{thm4.2}.
\end{proof}
Let $d_u$ denote the degree of a vertex $u$, and let $\overline{d}_S=|S|^{-1}\sum_{u\in S}d_u$ denote the average degree in a vertex set $S$. The largest Laplacian eigenvalue of a graph $G$ is the largest eigenvalue of the Laplacian matrix of $G$.

The following bound for independent sets was given in Theorem 4.3 of \cite{Godsil} without sufficient condition of the equality case.
\begin{cor}
\textup{\cite{Godsil}} Let $G$ be an $n$-vertex graph with the largest Laplacian eigenvalue $\mu>0$. For any independent set $S$ of $G$, we have
\begin{eqnarray*}
|S|\leq\frac{n(\mu-\overline{d}_S)}{\mu},
\end{eqnarray*}
with equality if and only if there exists a constant $d$ such that $d_u=d$ for each $u\in S$, and every vertex not in $S$ is adjacent to exactly $\mu-d$ vertices of $S$.
\end{cor}
\begin{proof}
Let $L$ be the Laplacian matrix of $G$. Since $(\mu I-L)e=\mu e$, we have $(\mu I-L,e)\in\mathcal{M}(G)$. By Lemma \ref{lem2.1}, we get
\begin{eqnarray*}
(\mu I-L)^\#e&=&\mu^{-1}e,\\
F_S(\mu I-L,e)&=&e^\top(\mu I-L)^\#e\sum_{u\in S}\frac{\mu-d_u}{|S|}=\frac{n}{\mu}(\mu-\overline{d}_S).
\end{eqnarray*}
The conclusion follows from part (1) of Theorem \ref{thm3.1}.
\end{proof}
The following bound was given independently in \cite{Lumei} and \cite{Godsil}, and the necessary condition of the equality case was given in Theorem 3.2 of \cite{Lumei}.
\begin{cor}\label{thm6.3}
\textup{\cite{Godsil,Lumei}} Let $G$ be an $n$-vertex graph with the largest Laplacian eigenvalue $\mu>0$ and the minimum degree $\delta$. Then
\begin{eqnarray*}
\alpha(G)\leq\frac{n(\mu-\delta)}{\mu},
\end{eqnarray*}
with equality if and only if there exists an independent set $C$ such that $d_u=\delta$ for each $u\in C$, and every vertex not in $C$ is adjacent to exactly $\mu-\delta$ vertices of $C$. Moreover, if an independent set $C$ satisfies the above conditions, then
\begin{eqnarray*}
\alpha(G)=\Theta(G)=\vartheta(G)=|C|=\frac{n(\mu-\delta)}{\mu}.
\end{eqnarray*}
\end{cor}
\begin{proof}
Let $L$ be the Laplacian matrix of $G$. Since $(\mu I-L)e=\mu e$, we have $(\mu I-L,e)\in\mathcal{M}(G)$. By Lemma \ref{lem2.1}, we get
\begin{eqnarray*}
(\mu I-L)^\#e&=&\mu^{-1}e,\\
F(\mu I-L,e)&=&(\mu-\delta)e^\top(\mu I-L)^\#e=\frac{n(\mu-\delta)}{\mu}.
\end{eqnarray*}
The conclusions follow from Theorems \ref{thm3.1} and \ref{thm4.2}.
\end{proof}
For a graph $G$ without isolated vertices, the \textit{normalized Laplacian matrix} $\mathcal{L}$ of $G$ is the $n\times n$ symmetric matrix with entries
\begin{eqnarray*}
(\mathcal{L})_{ij}=\begin{cases}1~~~~~~~~~~~~~~~~~~\mbox{if}~i=j,\\
-(d_id_j)^{-\frac{1}{2}}~~~~~~~\mbox{if}~\{i,j\}\in E(G),\\
0~~~~~~~~~~~~~~~~~~~\mbox{if}~\{i,j\}\notin E(G),\end{cases}
\end{eqnarray*}
The largest normalized Laplacian eigenvalue of $G$ is the largest eigenvalue of $\mathcal{L}$.

The following bound was given in \cite{Harant} without necessary and sufficient conditions of the equality case.
\begin{cor}\label{thm6.4}
\textup{\cite{Harant}} Let $G$ be a graph with $m$ edges, the largest normalized Laplacian eigenvalue $\mu$ and the minimum degree $\delta>0$. Then
\begin{eqnarray*}
\alpha(G)\leq\frac{2m(\mu-1)}{\mu\delta},
\end{eqnarray*}
with equality if and only if there exists an independent set $C$ such that $d_u=\delta$ for each $u\in C$, and every vertex $v$ not in $C$ is adjacent to exactly $(\mu-1)d_v$ vertices of $C$. Moreover, if an independent set $C$ satisfies the above conditions, then
\begin{eqnarray*}
\alpha(G)=\Theta(G)=\vartheta(G)=|C|=\frac{2m(\mu-1)}{\mu\delta}.
\end{eqnarray*}
\end{cor}
\begin{proof}
Let $\mathcal{L}$ be the normalized Laplacian matrix of $G$, and let $x=(d_1^{\frac{1}{2}},\ldots,d_n^{\frac{1}{2}})^\top$. Since $(\mu I-\mathcal{L})x=\mu x$, we have $(\mu I-\mathcal{L},x)\in\mathcal{M}(G)$. By Lemma \ref{lem2.1}, we get
\begin{eqnarray*}
(\mu I-L)^\#x&=&\mu^{-1}x,\\
F(\mu I-\mathcal{L},x)&=&\frac{\mu-1}{\delta}x^\top(\mu I-\mathcal{L})^\#x=\frac{2m(\mu-1)}{\mu\delta}.
\end{eqnarray*}
The conclusions follow from Theorems \ref{thm3.1} and \ref{thm4.2}.
\end{proof}
The following bound was given in Theorem 6.1 of \cite{Godsil} without necessary and sufficient conditions of the equality case.
\begin{cor}
\textup{\cite{Godsil}} Let $G$ be an $n$-vertex graph, and $M$ is a positive semidefinite matrix indexed by vertices of $G$ such that\\
(a) $(M)_{ij}\leq0$ if $i,j$ are two nonadjacent vertices.\\
(b) $M$ has constant row sum $r>0$.\\
(c) Every diagonal entry of $M$ is $t$.\\
Then
\begin{eqnarray*}
\alpha(G)\leq\frac{nt}{r},
\end{eqnarray*}
with equality if and only if $G$ has an independent set $C$ such that $M[C]$ is diagonal and $\sum_{u\in C}(M)_{vu}=t~(v\notin C)$. Moreover, if an independent set $C$ satisfies the above conditions, then
\begin{eqnarray*}
\alpha(G)=\vartheta'(G)=|C|=\frac{nt}{r}.
\end{eqnarray*}
\end{cor}
\begin{proof}
Since $Me=re$, we have $(M,e)\in\mathcal{P}(G)$. By Lemma \ref{lem2.1}, we get
\begin{eqnarray*}
M^\#e&=&r^{-1}e,\\
F(M,e)&=&te^\top M^\#e=\frac{nt}{r}.
\end{eqnarray*}
The conclusions follow from Theorems \ref{thm3.1} and \ref{thm5.3}.
\end{proof}

\section{Structural and spectral conditions}
Let $\delta(G)$ denote the minimum degree of a graph $G$. In the following theorem, we give some simple structural and spectral conditions for determining the independence number $\alpha(G)$, the Shannon capacity $\Theta(G)$ and the Lov\'{a}sz theta function $\vartheta(G)$. These conditions mean that if a given independent set $C$ satisfies some properties, then $C$ must be a maximum independent set in $G$, and $\Theta(G)=\vartheta(G)=|C|$.
\begin{thm}\label{thm7.1}
Let $G$ be a graph with an independent set $C$. Then
\begin{eqnarray*}
\alpha(G)=\Theta(G)=\vartheta(G)=|C|
\end{eqnarray*}
if $C$ satisfies one of the following conditions:\\
(1) Every vertex not in $C$ is adjacent to exactly $\lambda>-\tau$ vertices of $C$, where $\tau$ is the minimum adjacency eigenvalue of $G$.\\
(2) Every vertex not in $C$ is adjacent to exactly $-\tau$ vertices of $C$ and $G$ is regular, where $\tau<0$ is the minimum adjacency eigenvalue of $G$.\\
(3) $d_u=\delta(G)$ for each $u\in C$, and every vertex not in $C$ is adjacent to exactly $\mu-\delta(G)$ vertices of $C$, where $\mu>0$ is the largest Laplacian eigenvalue of $G$.\\
(4) $d_u=\delta(G)>0$ for each $u\in C$, and every vertex $v$ not in $C$ is adjacent to exactly $(\mu-1)d_v$ vertices of $C$, where $\mu$ is the largest normalized Laplacian eigenvalue of $G$.\\
(5) Every vertex not in $C$ is adjacent to all vertices of $C$ and $2|C|\geq|V(G)|-\delta(G-C)$.\\
(6) $G$ has a spanning subgraph $H$ ($\delta(H)>0$) which is semiregular bipartite with parameters $(n_1,n_2,r_1,r_2)$ ($n_1\leq n_2$), and $C$ is the independent set with $n_2$ vertices in $H$.\\
(7) There exists a clique covering $\varepsilon$ such that $d_v^\varepsilon=\min_{u\in V(G)}d_u^\varepsilon>0$ for each $v\in C$, and $|C\cap Q|=1$ for each $Q\in\varepsilon$.\\
(8) There exists a uniform hypergraph $H$ such that $G=\Omega(H)$ and $C$ corresponds to a perfect matching in $H$.\\
\end{thm}
\begin{proof}
Parts (1) and (2) follow from Corollaries \ref{thm4.5} and \ref{thm6.1}, respectively. Parts (3) and (4) follow from Corollaries \ref{thm6.3} and \ref{thm6.4}, respectively. Parts (5) and (6) follow from Examples 4.4 and 4.3, respectively. Parts (7) and (8) follow from Theorems \ref{thm4.7} and \ref{thm4.9}, respectively.
\end{proof}
A family $\mathcal{F}$ of $k$-subsets of $\{1,\ldots,n\}$ is called \textit{intersecting} if $F_1\cap F_2\neq\emptyset$ for any $F_1,F_2\in\mathcal{F}$. For $n\geq2k$, Erd\"{o}s, Ko and Rado \cite{Erdos} proved that the maximum size of intersecting families of $\{1,\ldots,n\}$ is $\binom{n-1}{k-1}$.

The \textit{Kneser graph} $K(n,k)$ is a graph whose vertices consist of all $k$-subsets of $\{1,\ldots,n\}$, and two vertices $U_1,U_2\subseteq\{1,\ldots,n\}$ are adjacent if and only if $U_1\cap U_2=\emptyset$. The Erd\"{o}s-Ko-Rado Theorem is equivalent to say that $\alpha(K(n,k))=\binom{n-1}{k-1}$ for $n\geq2k$. It is known \textup{[15, Theorem 9.4.3]} that the minimum adjacency eigenvalue of $K(n,k)$ is $-\binom{n-k-1}{k-1}$. By computation, the Hoffman bound for $K(n,k)$ is $\binom{n-1}{k-1}$. So we have $\alpha(K(n,k))=\binom{n-1}{k-1}$ because $K(n,k)$ has an  independent set of size $\binom{n-1}{k-1}$.

We can also obtain the EKR Theorem without using the Hoffman bound. All $k$-subsets containing a fixed vertex $1$ forms an independent set $C$ of size $\binom{n-1}{k-1}$, and every vertex not in $C$ is adjacent to exactly $\binom{n-k-1}{k-1}$ vertices of $C$. From part (2) of Theorem \ref{thm7.1}, we have
\begin{eqnarray*}
\alpha(K(n,k))=\Theta(K(n,k))=\vartheta(K(n,k))=|C|=\binom{n-1}{k-1}.
\end{eqnarray*}

\section{Concluding remarks}
Let $M^{\otimes k}$ denote the Kronecker product of $k$ copies of matrix $M$. If there is some $k$ such that $\alpha(G^k)=\vartheta(G)^k$, then $\Theta(G)=\vartheta(G)$. We can derive the following characterization of graphs satisfying $\alpha(G^k)=\vartheta(G)^k$.
\begin{thm}\label{thm8.1}
Let $(M,x)\in\mathcal{M}(G)$ be a matrix-vector pair for graph $G$ such that $\vartheta(G)=F(M,x)$. There is some $k$ such that $\alpha(G^k)=\vartheta(G)^k$ if and only if $G^k$ has an independent set $C$ such that $(M^{\otimes k},x^{\otimes k})$ and $C$ satisfy the conditions given in part (2) of Theorem \ref{thm3.1}. Moreover, if $G^k$ has an independent set $C$ satisfies the above conditions, then
\begin{eqnarray*}
\alpha(G^k)=\vartheta(G)^k=\Theta(G)^k=|C|.
\end{eqnarray*}
\end{thm}
\begin{proof}
Since $(M,x)\in\mathcal{M}(G)$, we have $(M^{\otimes k},x^{\otimes k})\in\mathcal{M}(G^k)$. By Theorem \ref{thm3.1}, we have
\begin{eqnarray*}
\alpha(G^k)\leq F(M^{\otimes k},x^{\otimes k})=F(M,x)^k=\vartheta(G)^k,
\end{eqnarray*}
with equality if and only if $G^k$ has an independent set $C$ such that $(M^{\otimes k},x^{\otimes k})$ and $C$ satisfy the conditions given in part (2) of Theorem \ref{thm3.1}. Moreover, if $G^k$ has an independent set $C$ satisfies the above conditions, then by Theorem \ref{thm4.2}, we have
\begin{eqnarray*}
\alpha(G^k)=\vartheta(G^k)=\Theta(G^k)=|C|=F(M,x)^k=\vartheta(G)^k.
\end{eqnarray*}
Since $\alpha(G^k)^{\frac{1}{k}}\leq\Theta(G)\leq\vartheta(G)$, we have $\Theta(G)=\vartheta(G)$.
\end{proof}
In \cite{Lovasz}, Lov\'{a}sz proved that $\Theta(C_5)=\sqrt{5}$. The cycle $C_5$ is also an example for Theorem \ref{thm8.1}.
\begin{example}
Let $A$ be the adjacency matrix of the cycle $C_5$, then $(A+2I,e)\in\mathcal{M}(C_5)$ and $((A+2I)^{\otimes 2},e^{\otimes 2})\in\mathcal{M}(C_5^2)$. Suppose that $\{0,1,2,3,4\}$ is the vertex set of $C_5$ and $E(C_5)=\{\{0,1\},\{1,2\},\{2,3\},\{3,4\},\{4,0\}\}$. Then $S=\{00,12,24,31,43\}$ is an independent set of $C_5^2$. By computation, we know that $((A+2I)^{\otimes 2},e^{\otimes 2})$ and $S$ satisfy the conditions given in part (2) of Theorem \ref{thm3.1}. By Theorem \ref{thm8.1}, we have $\alpha(C_5^2)=\vartheta(C_5)^2=\Theta(C_5)^2=5$.
\end{example}

\end{CJK*}
\end{spacing}
\end{document}